\documentclass{amsart}
\usepackage{amssymb}
\usepackage[all]{xy}

\newcommand{\Z}{\mathbf{Z}}
\newcommand{\Ps}{\mathbf{P}}

\newcommand{\C}{\mathbf{C}}

\newcommand{\cL}{\mathcal{L}}
\newcommand{\cF}{\mathcal{F}}

\newcommand{\cO}{\mathcal{O}}
\newcommand{\cE}{\mathcal{E}}
\newcommand{\cT}{\mathcal{T}}

\newtheorem{theorem}{Theorem}[section]
\newtheorem{lemma}[theorem]{Lemma}
\newtheorem{proposition}[theorem]{Proposition} 
\newtheorem{corollary}[theorem]{Corollary}
\newtheorem{conjecture}[theorem]{Conjecture}
\theoremstyle{definition}
\newtheorem{definition}[theorem]{Definition}
\newtheorem{notation}[theorem]{Notation}

\newtheorem{example}[theorem]{Example}
\theoremstyle{remark} 
\newtheorem{remark}[theorem]{Remark}

\DeclareMathOperator{\coker}{coker}

\DeclareMathOperator{\Pic}{Pic}

\DeclareMathOperator{\Sym}{Sym}
\DeclareMathOperator{\Cliff}{Cliff}

\title{Infinitesimal Torelli for elliptic surfaces revisited}
\author{Remke Kloosterman}
\address{Universit\`a degli Studi di Padova,
Dipartimento di Matematica,
Via Trieste 63,
35121 Padova, Italy}
\thanks{The author would like to thank Marian Aprodu for several discussions on Koszul cohomology of curves and Orsola Tommasi for various comments on a previous version.}
\begin{document}
\begin{abstract} 
In this article we give a new proof for the infinitesimal Torelli theorem for minimal elliptic surfaces without multiple fibers with Euler number at least 24 for nonconstant $j$-invariant. In the case of constant $j$-invariant we find a new proof in the case of Euler number at least 72. We also discuss several new counterexamples.
\end{abstract}
\maketitle

\section{Introduction}\label{secIntro}
Let $\pi:X\to C$ be a minimal elliptic surface without multiple fibers. 
There have been various results as to whether the infinitesimal Torelli property holds for $X$. If the geometric genus $p_g(X)$ vanishes then obviously the infinitesimal Torelli property does not hold. If $X$ is an elliptic K3 surface then it holds by the results on K3 surfaces.

The case $g(C)=0$, $p_g(X)\geq 2$ can be studied by techniques developed by Lieberman--Wilsker--Peters \cite{LPW} and Kii \cite{Kii}. These papers give a sufficient criterion for  infinitesimal Torelli for varieties with divisible canonical bundle. In the latter paper Kii proved infinitesimal Torelli for elliptic surfaces in case $g(C)=0$, $p_g(X)\geq 2$ and the $j$-invariant is nonconstant. In \cite{Ext} the author used a similar method to show that infinitesimal Torelli holds if again $g(C)=0$, $p_g(X)\geq 2$ holds, but the $j$-invariant is constant, and $\pi$ has at least $p_g(X)+3$ singular fibers. It is well known that an elliptic surface with $p_g(X)\geq 1, g(C)=0$ has at least $p_g(X)+2$ singular fibers. In \cite{Ext} it is also shown that elliptic surfaces with $p_g(X)\geq 2, g(C)=0$ and $p_g(X)+2$ singular fibers do not satisfy infinitesimal Torelli.
Chakiris \cite{Cha}, in his proof of the generic Torelli theorem for elliptic surfaces with a section and $g(C)=0$ obtained a generic infinitesimal Torelli result.

One easily checks that Kii's criterion cannot be applied in the case where the genus of $C$ is positive.  M.-H. Saito claimed in \cite{Sai} that infinitesimal Torelli holds for elliptic surfaces without multiple fibers such that   
$p_g(X)\geq 1$ and either the $j$-invariant is nonconstant or the $j$-invariant is constant, different from $0,1728$ and $\chi(\cO_X)\geq 3$ hold, but no constraints on the base curve.
However, Ikeda \cite{IkedaTor} recently obtained an example of an elliptic surface without multiple fibers, with nonconstant $j$-invariant and $p_g(X)=g(C)=1$, for which infinitesimal Torelli does not hold. 

 The original purpose of this paper was to give a new proof for infinitesimal Torelli for elliptic surfaces, by methods different from Saito's.  However, when preparing this paper, we learned that each of the steps in Saito's proof hold  under the additional  assumption that $\Omega^2_X$ is base point free.
 
 Our basic idea is to use Green's version of Kii's criterion \cite[Corollary 4.d.3]{GreenKos}. By doing so, we can reproduce Saito's result under the same hypothesis that $\Omega^2_X$ is base point free. The main difference with Saito's proof is that In the case of nonconstant $j$-invariant, the infinitesimal Torelli result is almost an immediate corollary from the Green-Kii criterion. In the constant $j$-invariant case our method covers many cases left out by Saito. In particular, we obtain results for elliptic surfaces with constant $j$-invariant 0 and 1728, and for elliptic fiber bundles which are not principal.
Moreover, our method yields a series of new counterexamples to infinitesimal Torelli.

To formulate our statement, we need a further invariant of the elliptic fibration $\pi:X\to C$.  Let $\cL$ be the dual of the line bundle $R^1\pi_*\cO_X$ on $C$ and let $d=\deg(\cL)$. It is well known that $d\geq 0$ and $d=0$ corresponds to the case of an elliptic fiber bundle. To apply  Green's version of  Kii's criterion, we need to check that $\Omega^2_X$ is base point free. It is easy to check that this happens  if $d>1$ or $d=1$ and $h^0(\cL)=0$.

 Moreover let $\Delta$ be the reduced effective divisor on $C$ whose support coincides with the support of the discriminant.
In this paper we prove two  results on classes of elliptic surfaces for which infinitesimal Torelli holds, one in the nonconstant $j$-invariant case and one in  the constant $j$-invariant case.
\begin{theorem} Let $\pi: X\to C$ be an elliptic surface without multiple fibers and with nonconstant $j$-invariant. If $d\neq 1$ or $d=1$ and $h^0(\cL)=0$ hold then $X$ satisfies infinitesimal Torelli.
\end{theorem}

\begin{theorem} Let $\pi: X\to C$ be an elliptic surface without multiple fibers, and with constant $j$-invariant. Suppose that 
\begin{enumerate}
\item if $g=0$ then $d>2$;
\item if $d=1$ then $h^0(\cL)=0$;
\item if $h^1(X)$ is odd then $\cL\not \cong \cO_C$.
\end{enumerate}
Then $X$ satisfies infinitesimal Torelli \emph{if and only if} the multiplication map
\[ \mu_\pi: H^0(\Omega^1_C\otimes \cL)\otimes H^0(\Omega^1_C\otimes \cL^{-1}(\Delta)) \to H^0((\Omega^1_C)^2 (\Delta))\]
is surjective.
\end{theorem}

We will comment a bit one the cases excluded.
If $g=0$ and $d\leq 2$ then, depending on $d$, we have product surfaces $E\times \Ps^1$ $(d=0)$, rational elliptic surfaces  $(d=1)$ or elliptic K3 surfaces $(d=2)$. For each of these cases it is well known whether infinitesimal Torelli does hold (K3) or does not hold (products and rational surfaces). If $h^1(X)$ is odd and $\cL\cong \cO_C$ then it is known that $X$ does not satisfy infinitesimal Torelli \cite[Section 8]{Sai} and we will come back to this in Section~\ref{secFiber}. Hence the only cases not covered by the above two theorems are the cases $d=1$ and $h^0(\cL)>0$, i.e, when $\Omega^2_X$ is not base point free.

The second theorem does not give a conclusive answer whether infinitesimal Torelli holds, but the map $\mu_\pi$ is studied extensively in the literature. 
For many cases we know that $\mu_\pi$ is surjective, which yields to the following corollary. Recall that if $d\geq 1$ then $s\geq d+1$.

\begin{corollary} Let $\pi: X\to C$ be an elliptic surface without multiple fibers and with constant $j$-invariant. Let $s$ be the number of singular fibers. Suppose that if $d=1$ then $h^0(\cL)=0$ holds. Moreover, Suppose that one of the following holds
\begin{enumerate}
\item $d\geq 6$;
\item $3\leq d\leq 5$ and $s\geq d+2$;
\item $d\in \{1,2\}$ and $s\geq d+3$;
\item  $d\in \{4,5\}$, $s= d+1$; $h^0(\cL^{-1}(\Delta))=0$ and $\Cliff(C)\geq 2$;
\item $d\in \{1,2,3\}$, $s= d+1$ and $h^0(\cL^{-1}(\Delta))=0$, either one of $\Omega^1_C\otimes \cL$ or $\Omega^1_C\otimes \cL(-\Delta)$ is very ample and  $\Cliff(C)\geq 2$;
\item  $d\in \{1,2\}$, $s=d+2$, $h^0(\cL^{-2}(\Delta))=0$;
\item $d=0$; $h^1(X)$ is even; $\cL\cong \cO_C$ and $C$ is not hyperelliptic.
\end{enumerate}
Then $X$ satisfies infinitesimal Torelli.
\end{corollary}

In some cases one can show that $\mu_\pi$ is not surjective.
\begin{theorem} Suppose $\pi:X\to C$ is an elliptic surface, with constant $j$-invariant and $d+1$ singular fibers. If $g=0$ then suppose additionally that $d\geq 3$ and if $d=1$ then suppose additionally that $h^0(\cL)=0$.  If $h^0(\cL^{-1}(\Delta))>0$ then $X$ does not satisfy infinitesimal Torelli
\end{theorem}

If $g=0$, $d\geq 2$ and $s=d+1$ then  $h^0(\cL^{-1}(\Delta))=2$ for degree reasons. Hence $X$ does not satisfy infinitesimal Torelli in this case. In this way we recover the counterexamples from \cite{Ext}. However, the result in that paper is much stronger. Namely, there we proved for $d\geq 3$ that the period map is constant on the locus of elliptic surfaces with constant $j$-invariant and $d+1$ singular fibers. However, the above result yields new counterexamples if the base curve has genus 1:

\begin{theorem} Suppose $\pi:X\to C$ is an elliptic surface, with $g(C)=1$, $d\in \{1,2\}$ and constant $j$-invariant different from $0,1728$. Then $X$ does not satisfy infinitesimal Torelli.
\end{theorem}

There are a few cases not covered by our results. For a certain number of classes of elliptic surfaces with $d\leq 5$ and constant $j$-invariant we do not know whether  $\mu_\pi$ is surjective or not. This is an extensively researched problem and we do not aim to elaborate on this.

A second class of surfaces  excluded are the surfaces with $d=1$ and $h^0(\cL)>0$, i.e., the  case where $\Omega^2_X$ has a one-dimensional base locus and our method breaks down. We expect that infinitesimal Torelli does not hold in this case and we will present evidence for this. Note that if moreover the $j$-invariant is constant then $\mu_\pi$ is not surjective in this case.
Also, Ikeda's counterexample is of this type.

Our strategy is to use Green's version of Kii's criterion. If $\Omega^n_X$ is base point free then this criterion reduces infinitesimal Torelli to a problem on the vanishing of a certain Koszul cohomology group. In the case of elliptic surfaces we can relate this Koszul cohomology group with a Koszul cohomology group on the base curve. If the $j$-invariant is nonconstant then it is easy to prove that this group vanishes, whereas if the $j$-invariant is constant then this group vanishes if and only if $\mu_\pi$ is surjective. This strategy leaves out the cases where $\Omega^2_X$ has a base locus, i.e.,  the case where $d=1, h^0(\cL)>0$; the case $g=d=0$ and some particular cases ($K3$ surfaces, nonalgebraic principal elliptic fiber bundles), because of some technicalities in the proof.

The paper is organised as follows. 
In Section~\ref{secPrelim} we recall some preliminaries on elliptic surfaces and on Koszul cohomology. In Section~\ref{secNcst} we prove the Torelli result for elliptic fibrations with nonconstant $j$-invariant. In Section~\ref{secCst} we prove the results for constant $j$-invariant such that $d>0$. 
 In Section~\ref{secFiber} we discuss the case $d=0$. Finally, in Section~\ref{secConj} we discuss what happens if $d=1$ and $h^0(\cL)>0$ hold.

\section{Preliminaries on elliptic surfaces and on Koszul cohomology}\label{secPrelim}
\begin{notation}
Let $\pi:X\to C$ be an elliptic surface without multiple fibers, but possibly without a section. Let $\cL$ be the dual of the line bundle $R^1\pi_*\cO_X$. (This is a line bundle, see \cite[(II.3.5)]{MiES} for the case of an elliptic fibration with a section or \cite{Sai} for the case of fibrations without multiple fibers.)
Let $d=\deg(\cL)$. 

If $\pi:X\to C$ is an elliptic fibration, let $S=\{P_1,\dots,P_s\}$ be the set of points of $C$ such that $\pi^{-1}(P_i)$ is singular. Let $\Delta=\sum_{P\in S} P$. Let $s=\deg(\Delta)$ be the number of singular fibers.  With $j(\pi):C\to \Ps^1$ we denote the morphism such that if $P\not \in S$ then $j(\pi)(P)$ is the $j$-invariant of $\pi^{-1}(P)$.
\end{notation}

We recall the following well known results, proofs of which can be found in \cite{MiES} in the case of (projective) elliptic surfaces with a section and in \cite{Sai} in the case of (complex analytic) elliptic surfaces without multiple fibers.

\begin{proposition}\label{propBasic}Let $\pi:X\to C$ be an elliptic surface without multiple fibers. Then the following properties hold
\begin{enumerate}
\item $d\geq 0$.
\item $d=0$ holds  if and only if $\pi$ is a fiber bundle.
\item $\Omega^2_X=\pi^*(\Omega^1_C\otimes \cL)$. In particular, if $\cL\not \cong \cO_C$ then $p_g(X)=g+d-1$.
\item If $j(\pi)$ is not constant then $\pi_*\Omega^1_X=\Omega^1_C$.
\item \label{exactseq} If $j(\pi)$ is constant then there is an exact sequence
\[ 0 \to \Omega^1_C\to \pi_*\Omega^1_X \to \cL(-\Delta)\to 0.\]
\item If $j(\pi)$ is constant then $s\geq \frac{6}{5}d$.
\end{enumerate}

\end{proposition}

\begin{corollary}\label{corBpf} Let $\pi:X\to C$ be an elliptic surface without multiple fibers and $d\geq 1$. Then $\Omega^2_X$ is base point free if and only if either $d\geq 2$ or
$d= 1$ and $h^0(\cL)=0$
holds.
\end{corollary}

\begin{proof} From $ \Omega^2_X=\pi^*(\Omega^1_C\otimes \cL)$ we obtain that $\Omega^2_X$ is base point free if and only if $\Omega^1_C\otimes \cL$ is base point free on $C$.

If $d\geq 2$ then $\deg(\Omega^1_C\otimes \cL)=2g-2+d\geq 2g$. Hence $\Omega^1_C\otimes\cL$ is  base point free for degree reasons.

Suppose now that $d=1$. 
The line bundle  $\Omega^1_C\otimes \cL$ has a base point if and only if there is a point $p\in C$ such that $h^0(\Omega^1_C \otimes \cL(-p))=h^0(\Omega^1_C \otimes \cL)$.
The right hand side equals $g$. The left hand side equals $h^1(\cL^{-1}(p))$. Since $\cL^{-1}(p)$ has degree 0, we have that $h^1(\cL^{-1}(p))=g$  if and only if $\cL^{-1}(p)\cong \cO_C$. This happens if and only if  $\cL$ is effective, i.e., if and only if $h^0(\cL)>0$. \end{proof}
\begin{remark}\label{remBps} If $d=1$, $g>0$  and $\Omega^2_X$ is not base point free then the base locus has codimension 1.
Moreover, if  $d=1$ and $h^0(\cL)=0$ then $g\geq 2$. 
\end{remark}

We now define Koszul cohomology groups.
\begin{definition}
Let $Y$ be a compact complex manifold.  Let $\cF$ be a coherent analytic sheaf on $Y$ and let $\cL$ be an analytic line bundle on $Y$. Then for any pair of integers $(p,q)$ we define the Koszul cohomology group $K_{p,q}(Y,\cF,\cL)$ as the cohomology of
\begin{small}
\[ H^0(\cF\otimes \cL^{ (q-1)}) \otimes \wedge^{p+1} H^0(\cL) \to H^0(\cF\otimes \cL^{ q}) \otimes  \wedge^p H^0(\cL) \to H^0(\cF\otimes \cL^{ (q+1)}) \otimes \wedge^{p-1} H^0(\cL).\]
\end{small}
If $\cF=\cO_Y$ then one write $K_{p,q}(Y,\cL)$ for $K_{p,q}(Y,\cO_Y,\cL)$.
\end{definition}

There is an obvious isomorphism $K_{p,q}(Y,\cF,\cL)\cong K_{p,q-1}(Y,\cF\otimes \cL,\cL)$. We will use this identification several times in our proofs.

\begin{lemma}\label{lemMor}
 Let $\pi:X\to Y$ be a morphism, $\cF$ a coherent sheaf on $X$, $\cL$ a line bundle on $Y$. Then for every $p,q$ we have that
\[ K_{p,q}(X,\cF,\pi^*\cL)\cong K_{p,q}(Y,\pi_*\cF,\cL).\]
\end{lemma} 
\begin{proof} By the projection formula we have
\[ \pi_*(\cF\otimes (\pi^*\cL)^q)\cong \pi_*\cF \otimes \cL^q.\]
In particular we have isomorphisms
\[ H^0(X,\cF\otimes (\pi^*\cL)^q) \cong H^0 (Y,\pi_*\cF \otimes \cL^q)
\mbox{ and } H^0(X,\pi^*\cL)\cong H^0(Y,\cL)\]
 and these isomorphisms are well behaved with respect to the differentials, hence
 \[ K_{p,q}(X,\cF,\pi^*\cL)\cong K_{p,q}(Y,\pi_*\cF, \cL).\]
\end{proof}

A crucial ingredient for our proofs is the duality theorem. We apply this theorem only in the case of curves. In this case the statement simplifies to
\begin{theorem}[Duality Theorem]\label{dualThm}
Let $C$ be a smooth projective curve. Let $\cL$ be a base point free line bundle on $C$ and $r=h^0(\cL)-1$. Then
\[ K_{p,q}(C,\cL)\cong K_{r-1-p,2-q}(C,\Omega^1_C,\cL)^*.\]
\end{theorem}
 For a proof see \cite[Theorem 2.c.6]{GreenKos}
 
 \begin{lemma} \label{lemVan} Let $\cL$ be a line bundle of degree $d$ on a smooth projective curve $C$.
 Then $ K_{d+g-3,1}(C,\Omega^1_C,\Omega^1_C\otimes \cL)=0$  if either $d\geq 2$ or $d=1$ and $h^0(\cL)=0$ hold.
 \end{lemma}
 \begin{proof}
 From Corollary~\ref{corBpf} it follows that $\Omega^1_C\otimes \cL$ is base point free, moreover $p_g(X)=h^0(\Omega^1_C\otimes \cL)=g+d-1$.  Hence we can apply  Theorem~\ref{dualThm} and we obtain
 \[ K_{d+g-3,1}(C,\Omega^1_C,\Omega^1_C\otimes \cL)\cong K_{0,1}(C,\Omega^1_C\otimes \cL)^*.\]
 We claim that the latter group is zero. Note that by definition this group  is the cokernel of the multiplication map
 \[ H^0(C,\cO) \otimes  H^0(C,\cL) \to H^0(C,\cL).\]
 which is obivously trivial.
\end{proof}

\section{Nonconstant $j$-invariant}\label{secNcst}

In the case of nonconstant $j$-invariants the proof of the infinitesimal Torelli theorem follows almost directly from Green's version of Kii's criterion, which we first recall for the reader's convenience.
\begin{theorem}[Kii-Green]\label{KiiThm}
Let $X$ be a compact K\"ahler manifold of dimension $n$. Suppose $\Omega^n_X$ is base point free. Let $p_g=h^0(\Omega^n_X)$. Then $X$ satisfies infinitesimal Torelli if and only if $K_{p_g-2,1}(X,\Omega^{n-1},\Omega^n)=0$.
\end{theorem}
For a proof of this theorem see \cite[Corollary 4.d.3]{GreenKos}.

\begin{theorem}\label{mainThm} Let $\pi: X \to C$ be an elliptic surface with nonconstant $j$-invariant and without multiple fibers, such that either $d\geq 2$ or both $d=1$ and $h^0(\cL)=0$ hold. Then $X$ satisfies infinitesimal Torelli.
\end{theorem}
\begin{proof}
Recall that by Corollary~\ref{corBpf} we have that $\Omega^2_X$ is base point free. 
Hence we may apply the Kii-Green criterion  Theorem~\ref{KiiThm} and it suffices to determine  whether $K_{p_g-2,1}(X,\Omega^1_X,\Omega^2_X)$ vanishes.

Using Lemma~\ref{lemMor} for the first isomorphism and various parts of Proposition~\ref{propBasic} for the second isomorphism we obtain 
 \[ K_{p_g-2,1}(X,\Omega^1_X,\Omega^2_X)\cong   K_{p_g-2,1}(C,\pi_*\Omega^1_X,\Omega^1_C\otimes \cL)\cong K_{d+g-3,1}(C,\Omega^1_C,\Omega^1_C\otimes \cL).\]
 The third group vanishes by Lemma~\ref{lemVan}.  Therefore $K_{p_g-2,1}(X,\Omega^1_X,\Omega^2_X)$ vanishes. Hence $X$ satisfies infinitesimal Torelli. 
 \end{proof}

\begin{remark}
If $j(\pi)$ is nonconstant then $d\geq 1$, hence the only case with nonconstant $j$-invariant and not covered by the above theorem is $d=1$ and $h^0(\cL)>0$. 

If $g=0$ and $d=1$ then $h^0(\cL)=2$ by Riemann-Roch. In this case $X$ is a rational elliptic surface and infinitesimal Torelli does not hold.
If $g=d=1$ then $h^0(\cL)=1$ by Riemann-Roch. The counterexamples of  Ikeda \cite{IkedaTor} are of this type. We discuss the case $d=1$ and $g>1$ in Section~\ref{secConj}.
\end{remark}

\begin{remark} In the proof we use Koszul duality on $C$. It is very crucial to work on $C$ rather than on $X$. Suppose $Y$ is an $n$-dimensional complex manifold and we would like to apply Koszul duality \cite[Theorem 2.c.6]{GreenKos} to $K_{p_g-2,1}(Y,\Omega_Y^{n-1},\Omega_Y^n)$. Then one of the hypothesis of this theorem is that 
$H^1(\Omega_Y^{n-1})$ vanishes. However, if this happens then the differential of the period map is zero anyway and we do not obtain any interesting statement.
\end{remark}

If $C$ is a hyperelliptic curve of genus $g$, then for any $d>1$ it is straightforward to construct an example of an elliptic surface over $C$ with $\deg(\cL)=d$ and nonconstant $j$-invariant. However, to construct an example with $\deg(\cL)=1$ and $h^0(\cL)=0$ (i.e, one satisfying infinitesimal Torelli) is more involved. We will now present such an example, for which we need the following basic result.

\begin{lemma}\label{lemTwist} Let $\cL$ be an effective line bundle of degree $1$ on a smooth projective curve $C$ of genus at least 2. Then there exists an element $\eta \in \Pic(C)[2]$ such that $\cL\otimes \eta$ is not effective.
\end{lemma}

\begin{proof} Suppose $p$ is such that $\cL\cong\cO_C(p)$. Since $g\geq 1$ we have that $p$ is unique. 

Pick now an $\eta\in \Pic^0(C)[2]\setminus \{0\}$. If $\cL\otimes \eta$ is effective then $\eta\cong \cO_C(q-p)$ for some $q\in C$. From $\eta^2\cong \cO_C$ it follows that $2p$ and $2q$ are linearly equivalent. Hence $C$ is hyperelliptic and both $p$ and $q$ are Weierstrass points. There are precisely $2g+1$ possibilities for $q$. Since $2^{2g}>2g+2$ for $g \geq 2$ there exists an $\eta$ which is not of this form and therefore $\cL\otimes \eta$ is not effective.
\end{proof}

\begin{proposition} Let $C$ be either a hyperelliptic curve of genus $g\geq 2$ or a trigonal curve of genus $g\geq 3$, with a point $p$ with ramification index 3. Then there exists a Jacobian elliptic surface $\pi:X\to C$ with $d=1$ and such that $X$ satisfies infinitesimal Torelli.
\end{proposition}

\begin{proof}
Suppose first that $C$  is a hyperelliptic curve and $p$ a Weierstrass point, let $\cL=\cO_C(p)$. Then $h^0(\cL)=1$, $h^0(\cL^2)=2$, $h^0(\cL^4)\geq 3$, $h^0(\cL^6)\geq 4$.

Pick now general elements $A\in H^0(\cL^4)$, $B\in H^0(\cL^6)$. Then with $(\cL, A,B)$ we can associate an elliptic surface $\pi': X'\to C$, and if $(A,B)$ is general then the $j$-invariant is nonconstant. (The $j$-invariant is constant if $A=0$ or $B=0$ or $A=\lambda G^2, B=\mu G^3$, for some $G\in H^0(\cL^2)$, these conditions define proper subvarieites of $H^0(\cL^4)\times H^0(\cL^6)$.)

We cannot apply Theorem~\ref{mainThm} to $\pi':X'\to C$  since $h^0(\cL)>0$. However, since $g>1$  there exists a line bundle $\eta$ such that $\eta^{\otimes 2}\cong\cO_C$ and $h^0(\cL\otimes \eta)=0$. Since $(\cL\otimes \eta)^{\otimes 2}=\cL^2$ we can identify $H^0(\cL^k)$ with $H^0((\cL\otimes\eta)^k)$ for $k=4,6$ and associate an elliptic surface $\pi:X\to C$ with the Weierstrass data $(\cL\otimes \eta, A,B)$. This surface has still $d=1$ but satisfies the hypothesis of our theorem and therefore satisfies infinitesimal Torelli.

A similar example exists if $C$ is trigonal and the degree 3 map has a point $p$ with ramification index $3$. If we take $\cL\cong \cO_C(p)$ then $h^0(\cL^6)>h^0(\cL^4)>0$ and we can apply the same reasoning.
\end{proof}

\begin{remark} By multiplying $\cL$ with a  line bundle of order 2 we considered a quadratic twist of the original fibration, i.e., there is an unramified degree two cover $\tilde{C}$ of $C$, such that the minimal smooth models of $X\times_C \tilde{C}$ and $X'\times_C\tilde{C}$ are isomorphic.
\end{remark}

\section{Constant $j$-invariant} \label{secCst}
In the case of constant $j$-invariant we obtain also an infinitesimal Torelli result, but in this case it does not directly follow from duality in Koszul cohomology.

In the sequel we have to exclude a few cases, namely $g=0, d\leq 2$ (Products $E\times \Ps^1$; rational elliptic surfaces and K3 sufaces); $d=1$ and $h^0(\cL)\neq 0$ and $d=0$. The case $d=0$ will be treated in Section~\ref{secFiber}.

\begin{theorem}\label{prpRedKos} Let $\pi: X \to C$ be an elliptic surface with constant $j$-invariant. Let us denote with $s$ the number of singular fibers. 
Assume that one of the following conditions holds
\begin{itemize}
\item $d\geq 3$; 
\item $d=2$ and $g>0$ ;
\item $d=1$ and $h^0(\cL)=0$.\end{itemize}
Then $X$ satisfies infinitesimal Torelli  if and only if the multiplication map
\begin{equation}\label{eqnCupCst}
\mu_{\pi} : H^0(C,\Omega^1_C\otimes \cL^{-1}(\Delta))\otimes H^0(C,\Omega^1_C\otimes \cL)\to H^0((\Omega^1_C)^2(\Delta))
\end{equation}
is surjective.
\end{theorem}

\begin{proof}
Recall that by Corollary~\ref{corBpf} we have that $\Omega^2_X$ is base point free. 
Hence we may apply Kii-Green criterion  Theorem~\ref{KiiThm} and it suffices to determine  whether $K_{p_g-2,1}(X,\Omega^1_X,\Omega^2_X)$ vanishes.
By Proposition~\ref{propBasic} we have that $p_g(X)=g+d-1$.
As in the proof of Theorem~\ref{mainThm}  we obtain
 \[ K_{p_g-2,1}(X,\Omega^1_X,\Omega^2_X)\cong   K_{p_g-2,1}(C,\pi_*\Omega^1_X,\Omega^1_C\otimes \cL)\cong K_{d+g-3,1}(C,\pi_*\Omega^1_X,\Omega^1_C\otimes \cL).\]
We will now calculate the right hand side.

For every $q\in \Z$ the short exact sequence from Proposition~\ref{propBasic} tensored with $(\Omega^1_C\otimes \cL)^q$ yields a long exact sequence, which starts with
\begin{equation}\label{shrtExa} 0\to H^0(\Omega^1_C \otimes (\Omega^1_C\otimes \cL)^q)\to H^0(\pi_*\Omega^1_X \otimes (\Omega^1_C\otimes \cL)^q) \to H^0(\cL(-\Delta)\otimes (\Omega^1_C\otimes \cL)^q).\end{equation}
We will now show that this is actually a short exact sequence, i.e., the final map is surjective.
From the existence of the long exact sequence it follows that the cokernel of the final map is a subspace of $H^1(\Omega^1_C \otimes (\Omega^1_C\otimes \cL)^q)$. 
Note that our assumptions imply that either $d\geq 3$; $d=2$ and $g\geq 1$ or $d=1$ and $g\geq 2$. Hence $\deg(\Omega^1_C\otimes \cL)=2g+d-2>0$. It follows now that for $q\geq 1$ the group $H^1(\Omega^1_C \otimes (\Omega^1_C\otimes \cL)^q)$ vanishes for degree reasons.

For $q\leq 0$, note that $s\geq d+1$ and $2g-2+d>0$. In particular,  we have that $\deg (\cL(-\Delta)\otimes (\Omega^1_C\otimes \cL)^q)= d-s +q(2g-2+d)<0$. Therefore $h^0(\cL(-\Delta)\otimes (\Omega^1_C\otimes \cL)^q)=0$ and the map is surjective.

By \cite[Corollary 1.d.4]{GreenKos} we can use the short exact sequence (\ref{shrtExa}) to obtain the following long exact sequence in Koszul cohomology 
\begin{eqnarray*}
 \dots \to K_{d+g-3,1}(C,\Omega^1_C,\Omega^1_C\otimes \cL) &\to&  K_{d+g-3,1}(C,\pi_*\Omega^1_X,\Omega^1_C\otimes \cL) \to  \\
  \to K_{d+g-3,1}(C,\cL(-\Delta),\Omega^1_C\otimes \cL)
 &\to& K_{d+g-4,2}(C,\Omega^1_C,\Omega^1_C\otimes \cL) \to  \dots \end{eqnarray*}
The group $K_{d+g-3,1}(C,\Omega^1_C,\Omega^1_C\otimes \cL)$ vanishes by Lemma~\ref{lemVan}. The group at the end, $K_{d+g-4,2}(C,\Omega^1_C,\Omega^1_C\otimes \cL)$, is dual to $K_{1,0}(C, \Omega^1_C\otimes \cL)$ by Theorem~\ref{dualThm}.
By definition, this group  is the cohomology of
\[ H^0((\Omega^1_C\otimes \cL)^{-1} ) \otimes \wedge^{2} H^0(\Omega^1_C \otimes \cL) \to H^0(\cO_C) \otimes   H^0(\Omega^1_C\otimes \cL) \to H^0(\Omega^1_C \otimes \cL)\otimes \C .\]
The group $H^0(C,(\Omega^1_C\otimes \cL)^{-1} )$ vanishes for degree reasons, whereas the second arrow is an isomorphism. Hence $ K_{1,0}(C, \Omega^1_C\otimes \cL)$ vanishes.
Therefore we have an isomorphism
\[ K_{d+g-3,1}(C,\pi_*\Omega^1_X,\Omega^1_C\otimes \cL)  \cong K_{d+g-3,1}(C,\cL(-\Delta),\Omega^1_C\otimes \cL).\]
Hence $X$ satisfies infinitesimal Torelli if and only if the latter group vanishes.

To obtain the final statement we can use Theorem~\ref{dualThm} to obtain
\[ K_{p_g-2,1}(C, \cL(-\Delta),\Omega^1_C\otimes \cL)^* \cong  K_{0,1}(C, \Omega^1_C\otimes \cL^{-1}(\Delta),\Omega^1_C\otimes \cL).\]
This latter group is the cokernel of the multiplication map
\[
 H^0(\Omega^1_C\otimes \cL^{-1}(\Delta))\otimes H^0(\Omega^1_C\otimes \cL) \to H^0((\Omega^1_C)^ 2(\Delta)).\]
\end{proof}

Multiplication maps of sections of line bundles have been extensively studied. We will now show that $\mu_\pi$ is surjective in many case.  For the first result we use  the $H^0$-lemma of Green.
\begin{lemma} \label{lemKosVanA} Let $\pi: X \to C$ be an elliptic surface without multiple fibers with constant $j$-invariant and $d\geq 1$. If $d=1$ then suppose that $h^0(\cL)=0$.
If one of the following holds 
\begin{enumerate}
\item  $d\geq 3$ and $s\geq d+2$;
\item  $d\in \{1,2\}$ and $s\geq d+3$;
\item $d\in \{1,2\}$, $s=d+2$ and $h^0(\cL^{-2}(\Delta))=0$;
\end{enumerate}
then the map $\mu_{\pi}$ is surjective.
\end{lemma}

\begin{proof}
In this proof we want to apply the  $H^0$--lemma \cite[Theorem 4.e.1]{GreenKos}. Applied to $\mu_\pi$ we find that if 
 $\Omega^1_C\otimes \cL$ is base point free and
\[ h^1 (\cL^{-2}(\Delta))\leq h^0(\Omega^1_C\otimes \cL)-2=g+d-3\]
holds then $\mu_\pi$ is surjective. The first condition holds because of our assumptions on $d$ and $\cL$. Hence we need to check the second condition.

Recall that $\deg(\cL^{-2}(\Delta))=s-2d$, and that in all our cases we assumed that at least $s\geq d+2$ holds.

Consider first the case with few singular fibers, i.e., suppose  $d+2\leq s \leq 2d-1$. Then $d\geq 3$ and  $\deg (\cL^{-2}\otimes(\Delta))<0$. Hence
\[ h^1 (\cL^{-2}\otimes(\Delta))= -\chi(\cL^{-2}(\Delta))=g-1-\deg(\cL^{-2}(\Delta))=g-1-s+2d\leq g-3+d\]
where the last inequality follows from $s\leq d+2$. Hence we covered this case.

Consider now the case $s\geq 2d$. Then $\cL^{-2}(\Delta)$ is a line bundle of nonnegative degree and hence
\[ h^1(\cL^{-2}(\Delta))\leq g.\]
For  $d\geq 3$ this is at most $g+d-3$, and hence the multiplication map is surjective. So we are left with the cases $d=1,2$.

If $d=2$ then we need to show that $h^1(\cL^{-2}(\Delta))\geq g-1$. However, $\cL^{-2}(\Delta)$ has nonnegative degree by assumption. The only line bundle of nonnegative degree whose $h^1$ equals at least $g$ is the trivial bundle, hence for $d=2$ and $\cO_C(\Delta)\not \cong\cL^2$ we have $h^1(\cL^{-2}(\Delta))\leq g-1 =g+d-3$. This finishes the case $d=2$.

If $d=1$ then we need to show $h^1(\cL^{-2}(\Delta))\leq g-2$. Again $\cL^{-2}(\Delta)$ has nonnegative degree. Recall that $h^0(\cL)=0$ forces $g\geq 2$.
The only line bundles with $h^1\geq g-1$ are line bundles of degree $0$ (i.e, $s=2$) or effective line bundles of degree $1$ (i.e, $s=3$ and $h^0(\cL^{-2}\Delta)>0$). We excluded these cases.
\end{proof}

Recall that for $d\geq 1$ we have $s\geq d+1$. Hence cases not covered by the previous lemma have $s=d+1$ or $s=d+2$. We can use  the results of \cite{ButMult} to show that $\mu_\pi$ is surjective for some cases with $s=d+1$.

\begin{lemma} \label{lemKosVanB} Let $\pi: X \to C$ be an elliptic surface without multiple fibers with constant $j$-invariant with $d+1$ singular fibers.
Assume that $g(C)\geq 2$ and $\Cliff(C)\geq 2$ or $g(C)\geq 3$ and $4-d\leq \Cliff(C)\leq 1$. If $d=1$ then assume $h^0(\cL)=0$. If $d\leq 2$ then assume that one of $\Omega^1_C \otimes \cL, \Omega^1_C\otimes \cL^{-1}(\Delta)$ is very ample.

If $h^0(\cL^{-1}(\Delta))=0$ then   $\mu_\pi$ is surjective.
 \end{lemma}
 
 \begin{proof}
Our assumptions on $\cL$ and $\cL^{-1}(\Delta)$ yield that both  $\Omega^1_C\otimes \cL$ and $\Omega_C^1\otimes \cL^{-1}(\Delta)$ are base point free. If $d\geq 3$ then the former line bundle is very ample. If $d\in \{1,2\}$ then at least one of the line bundles is very ample by assumption. In particular the image of $\mu_{\pi}$ separates points and tangents. This a requirement to apply the results of \cite{ButMult}.

If $\Cliff(C)\geq 2$ then
\[\deg(\Omega_C^1\otimes \cL)=2g-2+d\geq \deg(\Omega_C^1\otimes \cL^{-1}(\Delta))=2g-1\geq 2g+1-\Cliff(C)\]
holds. In this case it follows from \cite[Theorem 1]{ButMult}  that $\mu_\pi$ is  surjective.

If $\Cliff(C)\in \{0,1\}$ then $\Cliff(C)\geq 4-d$. 
In particular
\[ \deg(\Omega^1_C\otimes \cL^{-1}\otimes\Delta)+\deg(\Omega^1_C\otimes \cL) =4g-3+d \geq 4g+1-\Cliff(C)\]
holds. Hence if $g\geq 3$ then we can use \cite[Theorem 2]{ButMult} to conclude that $\mu_\pi$ is surjective.
\end{proof}

\begin{remark}
Since $\deg (\cL^{-1}(\Delta))=1$ the condition $h^0(\cL^{-1}(\Delta))=0$ implies $g\geq 2$. Hence we have to exclude genus 2 curves. Moreover the condition $\Cliff(C)\geq \min \{2, 4-d\}$ excludes curves with Clifford index $0$ (i.e., hyperelliptic curves) if $d\leq 3$ and curves with Clifford index 1 (trigonal curves and plane quintics) if $d\leq 2$. Hence for $g\geq 3, d\geq 4$ there are no  cases left open.
\end{remark}

We will now consider two cases not covered by the previous lemmata where $s=d+2$. In these cases $d\in \{1,2\}$.

\begin{lemma}  \label{lemKosVanC} Suppose $\pi: X\to C$ is an elliptic surface with $d=2$ and $\cL^2\cong \cO_C(\Delta)$. Suppose $\Cliff(C)\geq 1$. If $h^0(\cL^{-1}(\Delta))=0$ then $\mu_\pi$ is surjective.
\end{lemma}

\begin{proof}
For degree reasons the line bundle $\Omega^1_C\otimes \cL^{-1}(\Delta)$ is base point free. If this line bundle is not very ample then there exist points $p,q\in C$ such that
\[ \Omega_C^1\otimes \cL^{-1}(\Delta)(-p-q)\cong \Omega_C^1\]
Hence $\cL^{-1}(\Delta)\cong \cO_C(p+q)$. This contradicts $H^0(\cL^{-1}(\Delta))=0$. Hence $\Omega_C^1\otimes \cL^{-1}(\Delta)$ is very ample.
Since $\Cliff(C)\geq 1$ we have that 
\[ \deg(\Omega_C^1\otimes \cL^{-1}(\Delta))=2g\geq 2g+1-\Cliff(C)\]
Hence we can apply \cite[Theorem 1]{GreenLaz} to conclude that the multiplication map is surjective.
\end{proof}
\begin{remark}
The conditions $d=2$ and $\cL^2\cong \cO_C(\Delta)$ imply $s=4$, i.e., $s=d+2$.
\end{remark}

\begin{lemma}  \label{lemKosVanD} Suppose $\pi: X\to C$ is an elliptic surface with $d=1, s=3$, $h^0(\cL)=0$ and $h^0(\cL^{-2}(\Delta))>0$. 
Suppose $\Cliff(C)\geq 2$. 
If $h^0(\cL^{-1}(\Delta))=0$   then $\mu_\pi$ is surjective.
\end{lemma}

\begin{proof}
As in the previous proof we have that $\Omega_C^1\otimes \cL^{-1}(\Delta)$ is very ample and $\Omega_C^1\otimes \cL$ is base point free.
Since $\Cliff(C)\geq 2$ we have that 
\[ \deg(\Omega_C^1\otimes \cL)=2g-1\geq 2g+1-\Cliff(C)\]
Hence we can apply \cite[Theorem 1]{ButMult} to conclude that $\mu_\pi$ is surjective.
\end{proof}

\begin{remark} If $h^0(\cL)>0$ then $h^0(\cL^{-2}(\Delta))>0$ implies $h^0(\cL^{-1}(\Delta))>0$. I.e., in order to have the second group to be zero one needs $h^0(\cL)=0$.
\end{remark}

Finally we proceed with two cases where the multiplication map cannot be surjective.

\begin{lemma} \label{lemKosNonVanA}Suppose $\pi:X\to C$ is an elliptic surface with $d+1$ singular fibers, such that $h^0(\cL^{-1}(\Delta))>0$. Then  $\mu_\pi$ is not surjective.
\end{lemma}

\begin{proof}
Our assumptions imply that there is a point $p$ such that $\cL^{-1}(\Delta)\cong \cO_C(p)$. This point $p$ is a base point of $\Omega^1_C\otimes \cL^{-1}(\Delta)$ and a base point of the image of $\mu_\pi$.
Hence the image of $\mu_\pi$ is contained in $H^0(\Omega^1_C(\Delta)(-p))$, which is of dimension $g-1+d$, whereas $h^0(\Omega^1_C(\Delta))=g+d$. Hence $\mu_\pi$ is not surjective. 
\end{proof}

In the case $d=2$, $s=4$ and $\cL^2\cong \cO_C(\Delta)$ we use the following lemma to construct a counterexample:
\begin{lemma} 
\label{lemKosNonVanB} Let $C$ be an elliptic curve. Let $\pi: X \to C$ be an elliptic surfaces, with $d=2$. Assume $\cL^{2}\cong \cO_C(\Delta)$. Then $\mu_\pi$
is not surjective.
\end{lemma}

\begin{proof} In this case $\Omega^1_C\cong \cO_C$. Now $h^0(\cL)=2$ and $h^0(\cL^2)=4$ by Riemann-Roch. The multiplication map $\mu_\pi$ factors over $\Sym^2 H^0(\cL)$ which is three-dimensional hence the map is not surjective.
\end{proof}

\begin{theorem}\label{mainThmCst} Let $\pi: X \to C$ be an elliptic surface with constant $j$-invariant. Let $d=\deg(\cL)$ and $s$ the number of singular fibers. Assume that $d\geq 2$ or $d=1$ and $h^0(\cL)=0$.

If one of the following holds 
\begin{enumerate}
\item $g=0$ and $d=2$;
\item $s \geq d+3$;
\item $s=d+2$ and $d\geq 3$.
\item $s=d+1$; $h^0(\cL^{-1}(\Delta))=0$; $g\geq 3$ and $\Cliff(C)\geq \min\{4-d,2\}$. If $d\in \{1,2\}$ then one of $\Omega_C^1\otimes \cL$, $\Omega_C^1\otimes \cL^{-1}(
\Delta)$ is very ample.
\item $d\in \{1,2\}$; $s=d+2$; $h^0(\cL^{-2}(\Delta))=0$.
\item $d\in \{1,2\}$; $s=d+2$; $h^0(\cL^{-2}(\Delta))\neq 0$; $h^0(\cL^{-1}(\Delta))= 0$; $\Cliff(C)\geq 3-d$.
\end{enumerate}
then $X$ satisfies infinitesimal Torelli.
\end{theorem}

\begin{proof}
If $(g,d)=(0,2)$ then $X$ is a K3 surface and therefore satisfies infinitesimal Torelli. For all other case note that by Proposition~\ref{prpRedKos} it  suffices to check that  
$\mu_\pi$ is surjective. The second and third case follow  Lemma~\ref{lemKosVanA}, the other three cases from  Lemma~\ref{lemKosVanB}, ~\ref{lemKosVanC}, ~\ref{lemKosVanD}.
\end{proof}

\begin{remark}
In the case $(g,d)=(0,2)$ several steps in our proof do not hold anymore. E.g., the  sequence (\ref{shrtExa}) in the proof of  Proposition~\ref{prpRedKos} is not exact for several values of $q$ and therefore the long exact sequence in Koszul cohomology does not exist.
\end{remark}

We have also some counter examples to infinitesimal Torelli:
\begin{theorem}\label{ThmCounter} Let $\pi: X \to C$ be an elliptic surface with constant $j$-invariant.  Assume that $d\geq 2$ or $d=1$ and $h^0(\cL)=0$. If $d=2$ assume that $g(C)>0$.

\begin{enumerate}
\item If $s=d+1$ and $h^0(\cL^{-1}(\Delta))>0$ or
\item if $d=2$, $g=1$ and $\cO_C(\Delta)\cong\cL^2$ 
\end{enumerate}
then $X$ does satisfy infinitesimal Torelli.
\end{theorem}

\begin{proof}
If $s=d+1$ and  $h^0(\cL^{-1}(\Delta))>0$ then $ \mu_\pi$ is not surjective by Lemma~\ref{lemKosNonVanA}.
If $d=2$, $g=1$  and $\cO_C(\Delta)\cong\cL^2$ then $\mu_\pi$ is not surjective by Lemma~\ref{lemKosNonVanB}.

Hence it follows from Proposition \ref{prpRedKos} that infinitesimal Torelli does not hold for $X$.
\end{proof}

The following Corollary recovers the main result of  \cite{Ext}:
 \begin{corollary} Suppose $g\leq 1$ and $d\geq 3$. Let $\pi:X \to C$ be an elliptic fibration with constant $j$-invariant then $X$ satisfies infinitesimal Torelli if and only if $s>d+1$.
 \end{corollary}
 \begin{proof}
 If $s>d+1$ then $X$ satisfies infinitesimal Torelli by Theorem~\ref{mainThmCst}. 
 
 If $s=d+1$ then $\cL^{-1}(\Delta)$ has degree 1. Since $g\leq 1$ we have that $h^0(\cL^{-1}(\Delta))>0$, hence by Theorem~\ref{ThmCounter} $X$ does not satisfy infinitesimal Torelli.
 \end{proof}
 
\begin{corollary} Let $\pi: X\to C$ be an elliptic surface with constant $j$-invariant.
Suppose that one of 
\begin{enumerate}
\item $d=3$ and $j\neq 0, 1728$;
\item $d\in \{4,5\}$ and $j\neq 0$;
\item  $d\geq 6$
\end{enumerate}
holds. Then  $X$ satisfies infinitesimal Torelli.
\end{corollary}

\begin{proof}
If the $j$-invariant is different from 0, 1728 then $\pi$ has $2d$ singular fibers. For $d\geq 3$ this is at least $d+2$.

If the $j$-invariant is 1728 then $\pi$ has at least $\lceil \frac{4}{3}d \rceil$ singular fibers. This is at least $d+2$ for $d\geq 4$.

If the $j$-invariant is 0 then $\pi$ has at least $\lceil \frac{6}{5}d \rceil$ singular fibers. This is at least $d+2$ for $d\geq 6$.
\end{proof}

We will finish by showing that for every $g$ there exists an example of an elliptic surface with $d=5$ not satisfying infinitesimal Torelli. For this we need to construct an elliptic surface with  6 singular fibers, constant $j$-invariant and $h^0(\cL^{-1}(\Delta))=1$.
\begin{example}
Let $C$ be a curve of genus $g$, such that $C$ admits a morphism $f:C\to \Ps^1$ of degree $6$, which a single point over $\infty$ and $6$ points over $0$.
This implies that there is a $f\in K(C)$ such that $div(f)=P_1+P_2+P_3+P_4+P_5+P_6-6Q$ for appropriate distinct points $P_1,\dots,P_6,Q\in C$.

Now let $\cL=\cO_C(5Q)$ and let $A=0\in H^0(\cL^4)$ and $B=f^5\in H^0(\cL^6)$. Then the elliptic surface associated with $y^2z=x^3+Axz^2+Bz^3$ in $\Ps(\cE)$ has 6 $II^*$ fibers, namely over $P_1,\dots,P_6$. 
Moreover, $\cO_C(\Delta)\otimes \cL^{-1}=\cO_C(Q)$. Hence $s=6,d=5$ and $h^0(\cL^{-1}(\Delta))>0$. 

Hence infinitesimal Torelli does not hold for $X$ by Theorem~\ref{ThmCounter}. Examples of such a curve $C$ exist for every $g\geq 0$.
\end{example}

\section{Elliptic fiber bundle case}\label{secFiber}
In \cite[Section 7]{Sai} Saito discusses the infinitesimal Torelli problem for elliptic surfaces such that $\cL\cong \cO$, the case of \emph{principal} elliptic fiber bundles. 
In this section we discuss the period map in the case of non-principal bundles, i.e., when $d=0$ and  $\cL\not \cong \cO_C$. Then $\cL$ is a torsion bundle of order 2,3,4 or 6.
In this case  the relative dualizing sheaf is a line bundle and  we have an isomorphism $\omega_{X/C}\cong \pi^* \cL$. 

To study infinitesimal Torelli in this case one can use both the strategy of Section~\ref{secCst} as well as the approach taken in \cite[Section 7]{Sai}. It turns out that the latter approach yields  a stronger result.

\begin{theorem} Let $\pi: X \to C$ be an elliptic fiber bundle and suppose  that $\cL\not \cong \cO_C$. Then $X$ satisfies infinitesimal Torelli if and only if  the multiplication map
\[ \mu_\pi: H^0(\Omega^1_C\otimes \cL) \otimes  H^0(\Omega^1_C\otimes \cL ^{-1}) \to H^0((\Omega^1_C)^2)\]
is surjective. 
\end{theorem}
\begin{proof}
In the fiber bundle case we have that the relative dualizing sheaf is isomorphic to the sheaf of relative differentials, i.e., $\Omega^1_{X/C}\cong \omega_{X/C}\cong \pi^*\cL$. 
In particular we have a short exact sequence
\[ 0 \to \pi^* \Omega^1_C\to \Omega^1_X\to  \pi^*\cL \to 0.\]

Similarly as in the case of constant $j$-invariant $d>0$ we find that the following pieces of the long exact sequence of higher direct images
\begin{equation}\label{sesA} 0 \to \Omega^1_C\to \pi_* \Omega^1_X\to \cL\to 0\end{equation}
and
\[ 0 \to R^1\pi_*\pi^*\Omega^1_C\to R^1\pi_*\Omega^1_X \to R^1\pi_* \cL \to 0\]
are exact. 
Using the projection formula we obtain isomorphisms
\[ R^1\pi_*\pi^* \Omega^1_C\cong \Omega^1_C\otimes \cL^{-1}
\mbox{ and } R^1\pi_* \cL\cong \cL\otimes \cL^{-1}\cong \cO_C.\]
Therefore the second exact sequence simplifies to 
\begin{equation}\label{sesB} 0 \to \Omega^1_C \otimes \cL^{-1} \to R^1\pi_*\Omega^1_X \to \cO_C \to 0\end{equation}

As argued in \cite{Sai}, we have that $X$ satisfies infinitesimal Torelli if and only if the cup product map
\[ \mu: H^0(\Omega^2_X)\otimes H^1(\Omega^1_X)\to H^1(\Omega^1_X\otimes \Omega^2_X)\]
is surjective.

Using the Leray spectral sequence we find that $X$ satisfies infinitesimal Torelli if and only if
\[ \mu_1:H^0(C,\pi_*\Omega^2_X)\otimes H^1(\pi_*\Omega^1_X)\to H^1(\pi_*(\Omega^1_X\otimes \Omega^2_X))\]
and
\[ \mu_2:H^0(C,\pi_*\Omega^2_X)\otimes H^0(C,R^1\pi_*\Omega^1_X)\to H^0(R^1\pi_*\Omega^1_X\otimes \Omega^2_X)\]
are surjective.

Recall that $\Omega^2_X=\pi^* \Omega^1_C\otimes \cL$. Using the projection formula we obtain $\pi_*(\Omega^1_X\otimes \Omega^2_X)=\Omega^1_C \otimes \cL \otimes \pi_*\Omega^1_X$.
Tensor (\ref{sesA}) with $\Omega^1_C\otimes \cL$ and consider the following piece of the long exact sequence in cohomology:
\[ H^1(\Omega^1_C \otimes (\Omega^1_C \otimes \cL))\to H^1( \pi_*(\Omega^1_C\otimes (\Omega^1_C \otimes \cL)))\to H^1(\Omega^1_C\otimes \cL^2).\]
We claim that the first group is zero. Since $\cL$ has finite order, but is nontrivial we have that $g$ is at least $1$. If $g=1$ then $\Omega^1_C \otimes \Omega^1_C \otimes \cL$ is a nontrivial line bundle of degree zero and hence its first cohomology vanishes. If $g>1$ then the degree of $\Omega^1_C \otimes \Omega^1_C \otimes \cL$ equals $4(g-1)>2(g-1)$ and the first cohomology vanishes for degree reasons.

If $\cL$ has order at least 3 then also $H^1(\Omega_C^1\otimes \cL^2)$ is zero and therefore $\mu_1$ is surjective. On the other hand if $\cL^2\cong \cO_C$ then  $H^1(\Omega^1_C\otimes \cL^2)$ is onedimensional and it suffices to check whether the cup product map
\[ H^0(\Omega^1_C\otimes \cL)\otimes H^1(\cL) \to H^1(\Omega_C^1\otimes \cL^2)\]
is nontrivial. However this map coincides with Serre duality in this case and hence  $\mu_1$ is surjective.

To show that $\mu_2$ is surjective we consider this exact sequence (\ref{sesB})
and the sequence tensored  with $ \pi_*\Omega_X^2=\Omega^1_C\otimes \cL$. From 
\[H^1( \Omega^1_C \otimes \cL^{-1})=0=H^1((\Omega^1_C)^2)\]
it follows that both exact sequences split on sections and we can decompose the map $\mu_2$ in
\[ \mu_2^1: H^0(\Omega^1_C\otimes \cL^{-1}) \otimes H^0(\Omega^1_C\otimes \cL)\to H^0((\Omega^1_C)^2)\]
and
\[ \mu_2^2:H^0( \cO_C) \otimes H^0(\Omega^1_C\otimes \cL)\to H^0((\Omega^1_C)\otimes \cL)\]
and obtain an exact sequence
\[ \ker(\mu_2^2)\to \coker(\mu_2^1)\to\coker(\mu)\to \coker(\mu_2^2)\to 0.\]
The  map $\mu_2^2$ is obviously an isomorphism and $\mu_2^1$ is just  $\mu_\pi$. In particular, we obtain that $\coker(\mu_2)\cong \coker(\mu_\pi)$. Hence $\mu$ is surjective if and only if $\mu_\pi$ is surjective and therefore $X$ satisfies infinitesimal Torelli if and only if $\mu_\pi$ is surjective.
\end{proof}

\begin{remark}
If $\cL\cong \cO_C$ then Saito shows that if $h^1(X)$ is odd then $X$ does not satisfy infinitesimal Torelli whenever $g>1$, but does satisfy infinitesimal Torelli for $g=1$. If $h^1(X)$ is even $X$ then he shows that $X$ satisfies infinitesimal Torelli if $g=1$ or $g>1$ and $C$ is not hyerpelliptic.

In the case that $h^1$ is odd $\cL\cong \cO_C$ it turns out that the exact sequence (\ref{sesB}) does not split on sections. This turns  out to be an obstruction for the surjectivity of $\mu_2$ in this case and therefore for infinitesimal Torelli.
If $h^1$ is even then we can proceed as above, but one needs a small argument to show that  (\ref{sesB}) splits on sections, since $H^1( \Omega^1_C \otimes \cL^{-1})\neq 0$ in this case. This is precisely the approach by Saito.
\end{remark}

\begin{corollary}Let $\pi: X \to C$ be an elliptic fiber bundle and suppose $g=1$. If $\cL$ is nontrivial then $X$ does not satisfy infinitesimal Torelli.
\end{corollary}

\begin{proof}
If $\cL$ is nontrivial then
\[ H^1(\Omega^1_C\otimes \cL)=H^1(\cL)=0.\]
At the same time $h^1((\Omega^1)^2)=1$ hence $\mu_\pi$ is not surjective.
\end{proof}

\begin{corollary}Let $\pi: X \to C$ be an elliptic fiber bundle and suppose $g\geq 2$.
\begin{enumerate}
\item Suppose $h^1(X)$ is odd and $\cL\cong \cO_C$ then $X$ does not satisfy infinitesimal Torelli.
\item Suppose $h^1(X)$ is even or $\cL\not \cong \cO_C$. Then $X$ satisfies infinitesimal Torelli if and only if 
\[ \mu_\pi: H^0(\Omega^1_C\otimes \cL) \otimes  H^0(\Omega^1_C\otimes \cL ^{-1}) \to H^0((\Omega^1_C)^2)\]
is surjective. In particular if $\cL\cong \cO_C$ and $C$ is not hyperellitic then $X$ satisfies infinitesimal Torelli.
\end{enumerate}
\end{corollary}

\section{$\Omega^2_X$ not base point free}\label{secConj}
In this section we will focus on the case $\deg(\cL)=1$ and $h^0(\cL)>0$.

If $g=0$ then we know that infinitesimal Torelli does not hold, since $X$ is a rational surface. So we assume now that $g>0$. In particular  $\cL\cong \cO_C(p)$, for some  unique point $p\in C$.

In \cite{Sai} Saito considers at two occasions a multiplication map
\[ H^0(\Omega^2_C\otimes \cL)\otimes H^0(\cT) \to H^0(\cT)\]
for some torsion sheaf $\cT$. Saito reduces infinitesimal Torelli to the surjectivety of this map. This map is surjective if and only if the base locus of $\Omega^1_C\otimes \cL$ and the support of $\cT$ are disjoint. The latter happens if and only if $p$ is not in support in $\cT$.

However, the construction of $\cT$ is not sufficiently explicit to enable us to check this latter criterion.

If the $j$-invariant is constant there is further evidence. In this case $X$ is of product-quotient type, i.e., it is the quotient of a product $E\times \tilde{C}$, where $E$ is an elliptic curve, by a finite cyclic group $G$. However,  if $G$ has order at least 3 then one can invert the $G$-action on one of the factor and leave it invariant on the other, in order to obtain  some sort of dual surface, $\tilde{X}$. On easily checks that this duality interchanges the line bundles $\cL$ and $\cL^{-1}(\Delta)$, hence the multiplication map $\mu_{\pi}$ is the same map for both morphisms. If $d=1$ and $h^0(\cL)>0$ then the dual surface satisfies $s=d+1$ and $h^0(\cL^{-1}(\Delta))>0$, hence the dual surface does not satisfy infinitesimal Torelli.

In particular,  the map $\mu_\pi$ is not surjective. However in this case this is insufficient to determine the failure of infinitesimal Torelli.

\begin{conjecture}Let $\pi:X\to C$ be an elliptic surface with $d=1$ and $h^0(\cL)>0$. Then $X$ does not satisfy infinitesimal Torelli.
\end{conjecture}

\bibliographystyle{plain}
\bibliography{remke2}

\end{document}